\documentclass[11pt,a4paper]{amsart}
\usepackage[T1]{fontenc}
\usepackage[utf8]{inputenc}
\usepackage{lmodern,amsmath,amssymb,mathtools,microtype}
\usepackage[margin=27mm]{geometry}
\usepackage{enumitem}
\usepackage{booktabs,tabularx}
\usepackage[hidelinks]{hyperref}
\hypersetup{
 pdftitle={Strongly small-2 sets that are not Riesz sets on compact abelian groups},
 pdfauthor={Przemyslaw Ohrysko},
 pdfsubject={Spectral sets, singular measures and convolution smoothing},
 pdfkeywords={Riesz set, small-2 set, compact abelian group, Walsh group, Fourier--Stieltjes transform, singular measure, convolution, measure algebra}
}
\allowdisplaybreaks[2]
\numberwithin{equation}{section}
\newtheorem{theorem}{Theorem}[section]
\newtheorem{lemma}[theorem]{Lemma}
\newtheorem{fact}[theorem]{Fact}
\newtheorem{proposition}[theorem]{Proposition}
\newtheorem{corollary}[theorem]{Corollary}
\theoremstyle{definition}
\theoremstyle{remark}\newtheorem{remark}[theorem]{Remark}
\newcommand{\Z}{\mathbb Z}\newcommand{\N}{\mathbb N}
\newcommand{\C}{\mathbb C}\newcommand{\R}{\mathbb R}
\newcommand{\F}{\mathbb F}\newcommand{\T}{\mathbb T}
\newcommand{\E}{\mathbb E}
\newcommand{\supp}{\operatorname{supp}}
\newcommand{\Var}{\operatorname{Var}}\newcommand{\ord}{\operatorname{ord}}
\newcommand{\dd}{\,\mathrm d}\newcommand{\ind}{\mathbf 1}
\newcommand{\norm}[1]{\left\lVert#1\right\rVert}

\title[Strongly small-2 non-Riesz sets]{Strongly small-2 sets that are not Riesz sets on compact abelian groups}
\author{Przemys\l aw Ohrysko}
\address{University of Warsaw, Faculty of Mathematics, Informatics and Mechanics, Banacha 2, 02-097 Warsaw, Poland}
\email{P.Ohrysko@mimuw.edu.pl}
\date{}
\begin{document}
\begin{abstract}
We prove that every infinite compact abelian group $G$ admits a
countably infinite proper set $E\subset\widehat G$ that is strongly
small-2 but not Riesz. More precisely,
$|\alpha|*|\beta|$ is absolutely continuous for all
$\alpha,\beta\in M_E(G)$, whereas $M_E(G)$ contains a singular
complex measure of total mass one with exact Fourier support $E$
whose convolution square belongs to $L^2(G)$.
When the image of the doubling map on $\widehat G$ is infinite, $E$
may moreover be chosen so that
$\#(E\cap(\gamma-E))<\infty$ for every $\gamma$ and
$E\cap(-E)=\{0\}$. When that image is finite, we reduce the problem
to a nonrectangular construction on the Walsh group, for which
$E\cap(E+\gamma)$ is finite at every nonzero shift.
\end{abstract}
\maketitle
\enlargethispage{2pt} 

\section{Introduction}\label{sec:intro}
Let $G$ be a compact abelian group, let $m_G$ denote normalized Haar
measure, and write $\Gamma=\widehat G$ additively. For any abelian
group $A$ and integer $n\ge1$, we use the notation
\[
 nA:=\{na:a\in A\},\qquad A[n]:=\{a\in A:na=0\}.
\]
Thus $2\Gamma$ is the image of the doubling map on $\Gamma$, while
$\Gamma[2]$ is its kernel. We call an abelian group $A$
\emph{Boolean} if $2A=\{0\}$, equivalently if every element has
order at most two. A group is \emph{torsion} if every element has
finite order, and has \emph{bounded exponent} if $nA=\{0\}$ for
some integer $n\ge1$. A torsion group has \emph{unbounded exponent}
when no such integer exists. The Banach algebra $M(G)$ consists of
finite complex regular Borel measures on $G$, with convolution and
the total variation norm; we identify $L^1(G)$ with its closed ideal
of absolutely continuous measures. For $E\subseteq\Gamma$, set
\[
 M_E(G)=\{\mu\in M(G):\widehat\mu(\gamma)=0\ (\gamma\notin E)\},
 \qquad L^1_E(G)=M_E(G)\cap L^1(G).
\]
The set $E$ is a \emph{Riesz set} if every measure in $M_E(G)$ is
absolutely continuous, that is, if $M_E(G)=L^1_E(G)$. For an integer
$r\ge2$, we call $E$ \emph{small-$r$} if $\mu^{*r}\in L^1(G)$ for
every $\mu\in M_E(G)$. In particular, $E$ is small-2 precisely when
\begin{equation}\label{eq:small2}
 M_E(G)*M_E(G)\subseteq L^1_E(G),
\end{equation}
where the equivalence follows from polarization and
$\widehat{\alpha*\beta}=\widehat\alpha\,\widehat\beta$.

We use the following stronger variation-level notion. We say that $E$
is \emph{strongly small-2} if
\begin{equation}\label{eq:strong2}
 |\alpha|*|\beta|\ll m_G\qquad(\alpha,\beta\in M_E(G)).
\end{equation}
The total variation measures need not have Fourier support in $E$. Since
$|\alpha*\beta|\le |\alpha|*|\beta|$, condition \eqref{eq:strong2}
implies \eqref{eq:small2}. Thus a strongly small-2 set which supports
a singular measure gives, in particular, a small-2 set which is not
Riesz.

The F.~and M.~Riesz theorem gives the classical prototype of a Riesz
set: the nonnegative frequencies in the dual of the circle. The
abstract theory of Riesz sets was developed by Meyer \cite{Meyer}; see
also MacLean \cite{MacLean} for historical background. Singular
measures whose convolution squares are absolutely continuous go back
to Wiener and Wintner \cite{WienerWintner}. The uniform spectral
problem considered here is different: one asks for convolution to
smooth every measure whose Fourier--Stieltjes transform is supported
in a prescribed set. Criteria of this type were established by
Glicksberg \cite{Glicksberg}. Graham's theorem \cite{Graham} already
gives the variation-level conclusion $|\mu|*|\nu|\in L^1(G)$,
while Pigno \cite{Pigno} allows an exceptional Riesz set in the
finite-intersection condition. The notion of a small-$p$ set was introduced by Dressler, Parker and
Pigno in 1973 \cite{DresslerParkerPigno}. MacLean
\cite[p.~116]{MacLean} notes that Glicksberg's earlier
convolution-smoothing theorem provided the motivation for this
definition; see also Yamaguchi \cite{Yamaguchi} for further relations
with Riesz sets.

The implication ``small-2 $\Rightarrow$ Riesz'' also arises in the
study of Arens regularity of the ideals $L^1_E(G)$; see
\cite[p.~273]{Ulger} and \cite[the discussion after Figure~1 and Section~5]{EFG}.
The constructions below show that the implication fails, even under the
stronger variation-level condition
\eqref{eq:strong2}. The contribution here is the construction of
non-Riesz spectra satisfying the classical smoothing criteria,
together with singular witnesses whose Fourier supports equal those
sets and whose convolution squares satisfy quantitative bounds.

Our main result is the following.

\begin{theorem}\label{thm:main}
Let $G$ be an infinite compact abelian group. There exist a countably
infinite proper set $E\subsetneq\widehat G$ and a measure $\mu\in M(G)$
such that
\begin{equation}\label{eq:main}
 \begin{gathered}
 |\alpha|*|\beta|\ll m_G\quad(\alpha,\beta\in M_E(G)),\\
 \mu(G)=1,\qquad \mu\perp m_G,\qquad
 \supp\widehat\mu=E,\qquad \mu*\mu\in L^2(G).
 \end{gathered}
\end{equation}
Thus $E$ is strongly small-2 and is not a Riesz set.
\end{theorem}
The infinitude hypothesis is necessary: if $G$ is finite, then Haar
measure assigns positive mass to every point, so $M(G)=L^1(G)$ and
every spectral set is Riesz.

\begin{theorem}[The two spectral mechanisms]\label{thm:refined}
Let $G$ be an infinite compact abelian group and let $\Gamma=\widehat G$.
\begin{enumerate}[label=\textup{(\roman*)},leftmargin=*]
\item If $2\Gamma$ is infinite, the set and measure in
Theorem~\ref{thm:main} can be chosen with
\begin{equation}\label{eq:global-fi}
 \#(E\cap(\gamma-E))<\infty\quad(\gamma\in\Gamma),
 \qquad E\cap(-E)=\{0\}.
\end{equation}
\item If $2\Gamma$ is finite, then for every $\varepsilon>0$ the set
$E$ and the measure $\mu$ can be chosen so that $\mu$ is real and
\begin{equation}\label{eq:global-walsh}
 \begin{gathered}
 \#(E\cap(E+\gamma))<\infty\quad(\gamma\ne0),\\
 \norm\mu\le1+\varepsilon,\qquad
 \norm{\mu*\mu-m_G}_2^2\le\varepsilon.
 \end{gathered}
\end{equation}
In this case $E$ may be chosen inside a countable subgroup of
$\Gamma[2]$.
\end{enumerate}
Moreover, an infinite set satisfying the finite-intersection condition
in \eqref{eq:global-fi} exists only if $2\Gamma$ is infinite.
\end{theorem}

The two parts of Theorem~\ref{thm:refined} require different
spectral mechanisms. If $2\Gamma$ is infinite, we construct an
asymmetric set satisfying the finite-intersection condition in
\eqref{eq:global-fi}. If $2\Gamma$ is finite, we use a Boolean
subgroup of the dual. Every element of that subgroup is its own
negative, so the condition at $\gamma=0$ in \eqref{eq:global-fi}
would force $E$ to be finite. The second construction therefore
allows one exceptional shift and gives \eqref{eq:global-walsh}.

This dichotomy also has a topological interpretation. The identity
$\widehat{G/G[2]}=2\Gamma$ shows that $2\Gamma$ is finite exactly
when $G[2]$ is open. Equivalently, $2G$ is finite, since the
continuous doubling map induces a topological isomorphism
$G/G[2]\cong2G$. Thus the two cases are distinguished by the
presence or absence of an open Boolean subgroup of $G$.

Table~\ref{tab:models} summarizes the constructions. The first three rows
supply models for the case $2\Gamma$ infinite; their hypotheses need
not be mutually exclusive.
\begin{table}[t]
\centering
\caption{The construction models}\label{tab:models}
\small
\setlength{\tabcolsep}{5pt}
\renewcommand{\arraystretch}{1.13}
\begin{tabularx}{\textwidth}{@{}>{\raggedright\arraybackslash}p{0.43\textwidth}>{\raggedright\arraybackslash}X@{}}
\toprule
Situation in the dual & Construction and mechanism \\
\midrule
An element of infinite order, or torsion of unbounded exponent
& Rudin--Shapiro factors and bounded-coefficient separation
  (Section~\ref{sec:RS}). \\\addlinespace
A subgroup isomorphic to $\bigoplus_{\N}\Z/4\Z$
& Quadratic Boolean blocks with an order-four coordinate
  (Section~\ref{subsec:boolean-finite}). \\\addlinespace
A subgroup isomorphic to $\bigoplus_{\N}\Z/p\Z$, $p$ an odd prime
& Finite fields and the quadratic multiplicative character
  (Section~\ref{subsec:odd}). \\\addlinespace
$2\Gamma$ finite
& A Boolean quotient and frequency-dependent splitting into
  disjoint new spectra (Section~\ref{sec:walsh}). \\
\bottomrule
\end{tabularx}
\end{table}

In each construction the spectral condition and the singularity of
the witnessing measure are proved separately. The former gives
strong small-2 for the whole space $M_E(G)$; the latter is obtained
from separating second moments and convergence in total variation.
For the Rudin--Shapiro construction, pairwise moment factorization
replaces the independence available in genuine product groups.

\medskip
\noindent\textbf{Organization of the paper.}\par\nobreak
\noindent
Section~\ref{sec:criteria} collects the analytic and probabilistic
tools and proves the criteria used below. The first detailed construction, in
Section~\ref{sec:RS}, treats elements of infinite order and unbounded
torsion exponent. Section~\ref{sec:finite} supplies the remaining
asymmetric bounded-exponent models, and Section~\ref{sec:walsh}
treats the Walsh group. Section~\ref{sec:algebra} completes the proofs
of Theorems~\ref{thm:main} and~\ref{thm:refined}.

The author first learned of the small-2 versus Riesz problem from a
list of open problems prepared by Colin C.~Graham, which the author
received in 2016.

\subsection*{Notation and conventions}
For $\gamma\in\Gamma$, let $\chi_\gamma$ denote the corresponding
character of $G$. Our Fourier convention is
\[
 \widehat\mu(\gamma)=\int_G\overline{\chi_\gamma(x)}\dd\mu(x).
\]
Fourier transforms and averages on finite groups are normalized. If
$\Delta\le\Gamma$, its annihilator is
\[
 \Delta^\perp:=\{x\in G:\chi_\gamma(x)=1\text{ for every }\gamma\in\Delta\}.
\]
For a continuous homomorphism $q$, the notation $q_*$ denotes
push-forward of measures. We write $\nu\sim\rho$ when the positive
measures $\nu$ and $\rho$ are mutually absolutely continuous. For a
probability measure $\sigma$, the symbols $\E_\sigma$ and
$\Var_\sigma$ denote expectation and variance, respectively.

The notation $\supp\widehat\mu$ denotes the set of nonzero Fourier
coefficients in the discrete group $\Gamma$, not the spatial support
of $\mu$. The relation $\mu\perp m_G$ means $|\mu|\perp m_G$. For a
complex measure, $\overline\mu(B)=\overline{\mu(B)}$ denotes
conjugation without reflection. Thus
\[
 \widehat{\overline\mu}(\gamma)
   =\overline{\widehat\mu(-\gamma)},
 \qquad \mu^\star(B)=\overline{\mu(-B)}
\]
are different conventions. We write $m$ instead of $m_G$ when the
ambient group is clear, and identify an absolutely continuous measure
with its density when taking an $L^r$ norm. All product groups in the
explicit constructions are countable products.

\section{Analytic and probabilistic preliminaries}\label{sec:criteria}
The Fourier criteria control entire spectral spaces. The product-integration
and moment tools organize the explicit product constructions and their
singularity arguments, while the final norm lemma justifies passage to
limits.

\subsection{Fourier uniqueness and square summability}

\begin{fact}\label{lem:finitefourier}
If a measure on a compact abelian group has finite Fourier support,
then it has a trigonometric polynomial as its density with respect
to Haar measure.
\end{fact}

\begin{fact}[Square-summable Fourier coefficients]\label{lem:fourier-l2}
Let $G$ be a compact abelian group and let $\nu\in M(G)$.
If $\widehat\nu\in\ell^2(\widehat G)$, then $\nu$ has an $L^2(G)$
density and
\[
 \norm\nu_2^2=\sum_{\gamma\in\widehat G}|\widehat\nu(\gamma)|^2.
\]
In particular, if $\sum_\gamma|\widehat\nu(\gamma)|^4<\infty$, then
\begin{equation}\label{eq:fourth-power-criterion}
 \nu*\nu\in L^2(G),\qquad
 \norm{\nu*\nu}_2^2=\sum_\gamma|\widehat\nu(\gamma)|^4.
\end{equation}
If also $\nu(G)=1$, then
$\norm{\nu*\nu-m_G}_2^2=\sum_{\gamma\ne0}|\widehat\nu(\gamma)|^4$.
\end{fact}

\subsection{Mixed finite intersections}
We use the following mixed version of the classical finite-intersection
criterion of Glicksberg and Graham \cite{Glicksberg,Graham}. In the
equal-set case, the variation-level conclusion is already present in
Graham's theorem. We include a proof to fix the conjugation convention
and to treat the two spectral sets simultaneously.

\begin{proposition}[Mixed finite-intersection criterion]\label{prop:criterion}
Let $E_1,E_2\subseteq\widehat G$ satisfy
\begin{equation}\label{eq:finiteint}
 \#\bigl(E_1\cap(\gamma-E_2)\bigr)<\infty
       \qquad(\gamma\in\widehat G).
\end{equation}
Then
\[
 |\alpha|*|\beta|\ll m
 \qquad(\alpha\in M_{E_1}(G),\ \beta\in M_{E_2}(G)).
\]
In particular, \eqref{eq:finiteint} with $E_1=E_2=E$ implies that
$E$ is strongly small-2.
\end{proposition}
\begin{proof}
For $u,v,\xi\in\widehat G$,
\begin{equation}\label{eq:modulation}
 \widehat{(\chi_u\alpha)*(\chi_v\overline\beta)}(\xi)
 =\widehat\alpha(\xi-u)\,
    \overline{\widehat\beta(v-\xi)}.
\end{equation}
Its support is contained in
\[
 (u+E_1)\cap(v-E_2)
 =u+\bigl(E_1\cap((v-u)-E_2)\bigr),
\]
which is finite. Fact~\ref{lem:finitefourier} therefore gives a
trigonometric polynomial density for every convolution in
\eqref{eq:modulation}.

Write $\alpha=\phi|\alpha|$ and $\beta=\psi|\beta|$, where the polar
functions have modulus one almost everywhere for their respective
variations. For every finite positive regular Borel measure $\nu$ on $G$, the space
$C(G)$ is dense in $L^1(\nu)$, and trigonometric polynomials are
uniformly dense in $C(G)$. Consequently, we can choose polynomials
$P_n,Q_n$ such that
\[
 P_n\longrightarrow\overline\phi\text{ in }L^1(|\alpha|),
 \qquad Q_n\longrightarrow\psi\text{ in }L^1(|\beta|).
\]
It follows that
\[
 P_n\alpha\longrightarrow|\alpha|,
 \qquad Q_n\overline\beta\longrightarrow|\beta|
 \quad\text{in total variation norm}.
\]
Each $(P_n\alpha)*(Q_n\overline\beta)$ is a finite linear combination
of the convolutions in \eqref{eq:modulation}, hence belongs to $L^1(G)$.
Moreover,
\begin{align*}
 &\norm{(P_n\alpha)*(Q_n\overline\beta)-|\alpha|*|\beta|}\\
 &\quad\le
 \norm{P_n\alpha-|\alpha|}\,\norm{Q_n\overline\beta}
 +\norm{\alpha}\,\norm{Q_n\overline\beta-|\beta|},
\end{align*}
so the convolution inequality in $M(G)$ shows norm convergence to
$|\alpha|*|\beta|$. The subspace $L^1(G)$ is
norm closed in $M(G)$, which proves the claim.
\end{proof}

\subsection{Lifting from a quotient}
\begin{lemma}\label{lem:lift}
Let $\Delta\le\widehat G$, let $H=\Delta^\perp$, put $K=G/H$, and
let $q:G\to K$ be the quotient map. For $\nu\in M(G)$, its push-forward
$q_*\nu\in M(K)$ is defined by
\[
 (q_*\nu)(B)=\nu(q^{-1}(B))
 \qquad(B\subseteq K\text{ Borel}).
\]
Identifying $\widehat K$ with $\Delta$, there is a positive isometric
algebra homomorphism
$J:M(K)\to M(G)$ such that
\begin{gather}
 q_*J\sigma=\sigma,\qquad |J\sigma|=J|\sigma|,
 \label{eq:lift-variation}\\
 \widehat{J\sigma}(\gamma)=
 \begin{cases}
  \widehat\sigma(\gamma),&\gamma\in\Delta,\\
  0,&\gamma\notin\Delta.
 \end{cases}\label{eq:lift-fourier}
\end{gather}
Furthermore $J(fm_K)=(f\circ q)m_G$, with preservation of every
$L^r$ norm, $1\le r\le\infty$, and
$\sigma\perp m_K$ implies $J\sigma\perp m_G$.
For $E\subseteq\Delta$, the Riesz, small-2 and strongly small-2
properties are the same whether $E$ is considered in $\widehat K$
or in $\widehat G$.
\end{lemma}
\begin{proof}
For $\varphi\in C(G)$, averaging over $H$ defines
$A_H\varphi\in C(K)$ by
\[
 (A_H\varphi)(q(x))=\int_H\varphi(x+h)\dd m_H(h).
\]
Set $J=A_H^*$. Positivity and contractivity are immediate, while
$q_*J=I$ gives the reverse norm inequality. Averaging characters
gives \eqref{eq:lift-fourier}; Fourier uniqueness then proves
$J(\sigma*\tau)=J\sigma*J\tau$. Positivity gives
$|J\sigma|\le J|\sigma|$. The two positive measures have the same
mass, by isometry, which proves \eqref{eq:lift-variation}.

The quotient integration formula gives
$J(fm_K)=(f\circ q)m_G$ and equality of $L^r$ norms. Conversely,
$q_*$ takes absolutely continuous measures to absolutely continuous
measures, since the inverse image of a Haar-null Borel set is
Haar-null. If $|\sigma|$ is carried by a Haar-null set $B\subset K$,
then $J|\sigma|$ is carried by $q^{-1}(B)$; this proves preservation
of singularity. Finally, every $\nu$ whose Fourier support is
contained in $\Delta$ satisfies $\nu=Jq_*\nu$, by
\eqref{eq:lift-fourier} and Fourier uniqueness. Thus
\[
 J\bigl(M_E(K)\bigr)=M_E(G).
\]
The assertions about the full spectral spaces now follow from the
convolution, variation, and absolute-continuity identities.
\end{proof}

\subsection{Probability measures, moments, and product integration}
The material in this subsection is standard; we recall only the facts
and notation that will be used later. For a probability measure $\sigma$
and real square-integrable functions $X,Y$, write
\[
 \E_\sigma X=\int X\dd\sigma,\qquad
 \Var_\sigma(X)=\E_\sigma(X-\E_\sigma X)^2,
\]
and $\operatorname{Cov}_\sigma(X,Y)=\E_\sigma[(X-\E_\sigma X)
(Y-\E_\sigma Y)]$. Expanding the square gives
\begin{equation}\label{eq:variance-sum}
 \Var_\sigma\left(\sum_{j=1}^N X_j\right)
 =\sum_{j=1}^N\Var_\sigma(X_j)
   +2\sum_{1\le i<j\le N}\operatorname{Cov}_\sigma(X_i,X_j).
\end{equation}
The family $(X_j)$ is \emph{independent} if
$\sigma(\bigcap_{j\in J}\{X_j\in B_j\})
=\prod_{j\in J}\sigma(X_j\in B_j)$ for every finite index set $J$
and Borel sets $B_j\subseteq\R$. Independence implies zero pairwise
covariances; the latter condition alone, called \emph{pairwise
uncorrelatedness}, suffices for variances to add.

\begin{lemma}[Countable products and integration]\label{lem:probability-product}
Let $(H_j,\sigma_j)$ be finite discrete probability spaces. On the compact
product $H=\prod_{j\ge1}H_j$ there is a unique regular Borel
probability measure $\sigma=\bigotimes_{j\ge1}\sigma_j$ with the
prescribed finite-coordinate marginals. The coordinates are
independent, and for every finite index set $J\subset\N$ and
functions $v_j:H_j\to\C$, $j\in J$, one has
\begin{equation}\label{eq:probability-factorization}
 \int_H\prod_{j\in J}v_j(x_j)\dd\sigma(x)
 =\prod_{j\in J}\int_{H_j}v_j\dd\sigma_j.
\end{equation}
If $\sigma_j=p_jm_j$, its first $N$-coordinate marginal has density
$\prod_{j=1}^N p_j$ relative to $\bigotimes_{j=1}^N m_j$.
\end{lemma}
\begin{proof}
Integration with the finite product measures defines a positive
norm-one functional on continuous functions depending on finitely
many coordinates. Such functions are uniformly dense in $C(H)$.
The functional therefore extends uniquely to $C(H)$, and the Riesz
representation theorem gives $\sigma$. Finite-product integration
proves \eqref{eq:probability-factorization}, the marginal formula,
and independence, first on cylinder sets and then for coordinate
functions.
\end{proof}

For nonnegative measurable functions $Y_j$, monotone convergence gives
\begin{equation}\label{eq:tonelli-summability}
 \int\sum_jY_j\dd\sigma=\sum_j\int Y_j\dd\sigma.
\end{equation}
If the right-hand side is finite, then $\sum_jY_j<\infty$
$\sigma$-almost everywhere.

\subsection{Concentration and a criterion for mutual singularity}
For a probability measure $\sigma$ and a real $X\in L^2(\sigma)$,
Chebyshev's inequality states that
\[
 \sigma\bigl(|X-\E_\sigma X|>t\bigr)
 \le \frac{\Var_\sigma(X)}{t^2}\qquad(t>0).
\]
For measurable sets $A_\ell$, the first Borel--Cantelli lemma states
that $\sum_\ell\sigma(A_\ell)<\infty$ implies
$\sigma(\limsup_\ell A_\ell)=0$, where
$\limsup_\ell A_\ell=\bigcap_n\bigcup_{\ell\ge n}A_\ell$.
Thus almost every point belongs to only finitely many $A_\ell$.
These facts give the following separation criterion.

\begin{lemma}[Separation by second moments]\label{lem:second-moment}
Let $\sigma$ and $\tau$ be probability measures on the same
measurable space. Let $T_N$ be real functions in both $L^2(\sigma)$
and $L^2(\tau)$, and let $v_N>0$ tend to infinity. Suppose that,
for a constant $C<\infty$ and distinct real numbers $a,b$,
\begin{equation}\label{eq:second-moment-hypotheses}
 \begin{gathered}
 \max\{\Var_\sigma(T_N),\Var_\tau(T_N)\}\le Cv_N,\\
 \frac{\E_\sigma T_N}{v_N}\longrightarrow a,\qquad
 \frac{\E_\tau T_N}{v_N}\longrightarrow b.
 \end{gathered}
\end{equation}
Then $\sigma\perp\tau$. More precisely, along any increasing
sequence $(N_\ell)$ with $v_{N_\ell}\ge2^\ell$,
\[
 \frac{T_{N_\ell}}{v_{N_\ell}}\longrightarrow a
       \quad\sigma\text{-a.e.},\qquad
 \frac{T_{N_\ell}}{v_{N_\ell}}\longrightarrow b
       \quad\tau\text{-a.e.}
\]
\end{lemma}
\begin{proof}
For $\lambda\in\{\sigma,\tau\}$ and $r\in\N$,
Chebyshev's inequality gives
\[
 \lambda\left(
 \left|\frac{T_{N_\ell}-\E_\lambda T_{N_\ell}}{v_{N_\ell}}\right|
 >\frac1r\right)
 \le\frac{Cr^2}{v_{N_\ell}}\le Cr^2 2^{-\ell}.
\]
For each $r$ this is summable in $\ell$. Apply the first
Borel--Cantelli lemma for all $r\in\N$ and intersect
the resulting countably many full-measure sets. The centered
normalized functions converge to zero almost everywhere under
either measure. The mean limits in
\eqref{eq:second-moment-hypotheses} give the asserted limits.
The set where the limit is $a$ has full $\sigma$-measure and zero
$\tau$-measure, since $a\ne b$. This proves mutual singularity.
\end{proof}

\begin{corollary}[A separating statistic]\label{lem:statistic}
Let $m=\bigotimes_jm_j$ and $\rho=\bigotimes_j(p_jm_j)$ on a
countable product of finite probability spaces, where each $p_jm_j$
is a probability measure. Suppose that $u_j:H_j\to\R$ are uniformly
bounded and satisfy
\[
 \E_{m_j}u_j=0,\qquad
 \frac1N\sum_{j\le N}\E_{p_jm_j}u_j\longrightarrow c\ne0.
\]
Then $\rho\perp m$.
\end{corollary}
\begin{proof}
Regard $u_j$ as coordinate functions and put
$M=\sup_j\norm{u_j}_\infty$. By
Lemma~\ref{lem:probability-product} and \eqref{eq:variance-sum},
under either $\lambda=m$ or $\lambda=\rho$,
\[
 \Var_\lambda\left(\sum_{j\le N}u_j\right)
 =\sum_{j\le N}\Var_\lambda(u_j)\le M^2N.
\]
Apply Lemma~\ref{lem:second-moment} with
$T_N=\sum_{j\le N}u_j$, $v_N=N$, $a=0$, and $b=c$.
\end{proof}

\subsection{Summable perturbations in total variation}
\begin{lemma}\label{lem:variation-limit}
Let $(\nu_N)$ be finite complex measures on a compact space and
suppose that $\sum_{N\ge0}\norm{\nu_{N+1}-\nu_N}<\infty$.
Then there is a measure $\nu$ such that
\[
 \norm{\nu-\nu_N}\le\sum_{j\ge N}\norm{\nu_{j+1}-\nu_j}
 \longrightarrow0.
\]
If every $\nu_N$ is absolutely continuous with respect to one fixed
positive measure $\rho$, then so is $\nu$. If every $\nu_N$ is
singular with respect to one fixed positive measure $m$, then so
is $\nu$.
\end{lemma}
\begin{proof}
The summable increments make $(\nu_N)$ a Cauchy sequence in the
Banach space of measures, and telescoping gives the bound. For a
$\rho$-null set $B$, the inequality
$|\nu|(B)\le\norm{\nu-\nu_N}+|\nu_N|(B)$ proves absolute
continuity in the limit. For singularity, choose $m$-null carriers
$B_N$ for $\nu_N$. Their countable union is still $m$-null, and
the same inequality shows that it carries $\nu$.
\end{proof}

For a fixed positive measure $\rho$, the identity
$\norm{Z\rho}=\norm{Z}_{L^1(\rho)}$ identifies convergence of
densities in $L^1(\rho)$ with convergence of their measures in
total variation. By contrast, weak-$*$ convergence, meaning
convergence of integrals against every continuous function, does
not in general preserve singularity.

\section{Elements of infinite order and unbounded exponent}\label{sec:RS}
We first give a detailed construction when the dual contains an
element of infinite order or is torsion of unbounded exponent.
The circle is the basic example of the first case. Rudin--Shapiro
polynomials provide asymmetric local spectra, while separation of
characters replaces an independent product decomposition of the
ambient group.

\begin{theorem}[Rudin--Shapiro construction]\label{thm:RS}
Suppose that $\widehat G$ contains an element of infinite order, or
that $\widehat G$ is a torsion group of unbounded exponent. Then there
are $E$ and $\mu$ satisfying \eqref{eq:main} and
\eqref{eq:global-fi}.
\end{theorem}

The use of complex measures is forced by the asymmetry of $E$.
Indeed, if $E\cap(-E)=\{0\}$ and $\sigma\in M_E(G)$ is real,
then $\widehat\sigma(-\gamma)=\overline{\widehat\sigma(\gamma)}$
forces all nonconstant coefficients to vanish. Hence $\sigma$ is a
real multiple of Haar measure, and any nonzero singular witness
with this spectrum must be genuinely complex. This explains the
complex perturbations used here and in Section~\ref{sec:finite};
the Walsh construction in Section~\ref{sec:walsh} instead produces
real signed measures.

\subsection{The trigonometric factors}
We use the Rudin--Shapiro recursion, recalled for example in
\cite[Section~1]{Erdelyi}:
\[
 A_0=B_0=1,\qquad
 A_{r+1}(z)=A_r(z)+z^{2^r}B_r(z),\qquad
 B_{r+1}(z)=A_r(z)-z^{2^r}B_r(z).
\]
Inductively, both polynomials have coefficients in $\{-1,1\}$ and
degree $2^r-1$. The parallelogram identity gives
\begin{equation}\label{eq:RS-recursion}
 |A_r(z)|^2+|B_r(z)|^2=2^{r+1}\qquad(|z|=1).
\end{equation}
Fix
\[
 a=\frac14,\qquad N_j=2^{4j},\qquad c=\frac{a^2}{2},
 \qquad zA_{4j}(z)=\sum_{n=1}^{N_j}\epsilon_{j,n}z^n.
\]
On $\T=\R/(2\pi\Z)$, with $\dd m_\T(t)=\dd t/(2\pi)$, set
\begin{align}
 g_j(t)&=\frac a{\sqrt{N_j}}
       \sum_{n=1}^{N_j}\epsilon_{j,n}\cos(nt),\notag\\
 h_j(t)&=\frac a{\sqrt{N_j}}
       \sum_{n=1}^{N_j}\sin(nt),\notag\\
 p_j&=1+g_j,\qquad f_j=p_j+ih_j.
 \label{eq:RS-factors}
\end{align}
By \eqref{eq:RS-recursion} and trigonometric orthogonality,
\begin{gather}
 \norm{g_j}_\infty\le a\sqrt2<\frac12,\qquad p_j\ge\frac12,
 \label{eq:RS-positive}\\
 \int g_j\dd m_\T=\int h_j\dd m_\T
 =\int g_jh_j\dd m_\T=0,\qquad
 \int g_j^2\dd m_\T=c.\label{eq:RS-moments}
\end{gather}
For $0<|t|\le\pi$, the geometric-series formula yields
\[
 \left|\sum_{n=1}^N\sin(nt)\right|
 \le\left|\sum_{n=1}^Ne^{int}\right|
 \le\min(N,\pi/|t|).
\]
Integrating separately on $(0,\pi/N)$ and $(\pi/N,\pi)$ shows that
\begin{equation}\label{eq:RS-cost}
 \norm{h_j}_1\le d_j:=\frac{a(1+\log N_j)}{\sqrt{N_j}},
 \qquad \sum_jd_j<\infty.
\end{equation}
Since $p_j\ge0$, we also have
\[
 \int f_j\dd m_\T=1,\qquad |f_j|\le p_j+|h_j|,
 \qquad\norm{f_j}_1\le1+d_j.
\]
For $1\le n\le N_j$, direct calculation gives
\begin{equation}\label{eq:RS-coeff}
 \widehat f_j(n)=\frac a{2\sqrt{N_j}}(\epsilon_{j,n}+1),
 \qquad
 \widehat f_j(-n)=\frac a{2\sqrt{N_j}}(\epsilon_{j,n}-1).
\end{equation}
Thus the exact support and the nonzero coefficients are
\begin{gather}
 S_j=\{0\}\cup\{\epsilon_{j,n}n:1\le n\le N_j\},
 \qquad S_j\cap(-S_j)=\{0\},\label{eq:RS-local-support}\\
 b_j(0)=1,\qquad
 b_j(\epsilon_{j,n}n)=\frac{a\epsilon_{j,n}}{\sqrt{N_j}}.
 \label{eq:RS-values}
\end{gather}
The real parts have a fixed nonzero variance, while removing half
of the nonconstant frequencies costs a summable amount in $L^1$.

\subsection{Polynomial majorants and bounded-coefficient separation}
By uniform approximation, choose nonnegative real trigonometric
polynomials $U_j,V_j$ satisfying
\begin{equation}\label{eq:RS-majorants}
 \begin{gathered}
 U_j\ge|f_j|,\qquad V_j\ge|h_j|,\\
 \int U_j\dd m_\T\le1+\delta_j,\qquad
 \int V_j\dd m_\T\le\delta_j,\qquad
 \delta_j=d_j+2^{-j}.
 \end{gathered}
\end{equation}
For example, approximate each modulus by a real polynomial with
uniform error at most $2^{-j-1}$ and add that constant to the
approximant. Let
\[
 D_j\ge\max(N_j,\deg U_j,\deg V_j),\qquad D_j\in\N,
\]
where degree bounds the absolute values of the Fourier indices.
We seek characters $\gamma_j\in\Gamma$ satisfying
\begin{description}[leftmargin=2.5em,style=nextline]
\item[\textup{(D)}] If a finitely supported integer sequence satisfies
$|k_j|\le4D_j$ and $\sum_jk_j\gamma_j=0$, then all $k_j$ vanish.
\end{description}
Condition (D) implies that each $\gamma\in\Gamma$ has at most one
finitely supported representation
$\gamma=\sum_jk_j\gamma_j$ with $|k_j|\le2N_j$: subtracting two
such representations gives coefficients bounded by $4N_j\le4D_j$.
We shall use the following consequence:
\begin{description}[leftmargin=2.5em,style=nextline]
\item[\textup{(U)}] For each $\gamma\in\Gamma$ there is $J(\gamma)$
such that every finite representation
$\gamma=\sum_jk_j\gamma_j$ with $|k_j|\le2N_j$ has all its nonzero
indices at most $J(\gamma)$.
\end{description}
Here $J(\gamma)$ is the largest nonzero index in the unique
representation, and is zero if that representation is identically
zero or does not exist.

\begin{lemma}\label{lem:RS-separation}
The characters can be chosen to satisfy \textup{(D)} under either
hypothesis of Theorem~\ref{thm:RS}.
\end{lemma}
\begin{proof}
First suppose that $\gamma_0$ has infinite order. Put
\begin{equation}\label{eq:RS-lambda}
 \gamma_j=\lambda_j\gamma_0,\qquad
 \lambda_1=1,\qquad
 \lambda_j=1+4\sum_{i<j}D_i\lambda_i.
\end{equation}
For a last nonzero coefficient $k_r$ with $|k_i|\le4D_i$,
\[
 \left|\sum_i k_i\lambda_i\right|
 \ge\lambda_r-4\sum_{i<r}D_i\lambda_i=1,
\]
which proves (D). For $G=\T$, take $\gamma_0=1$, so
\eqref{eq:RS-lambda} gives an explicit sequence of integer dilations.

Now suppose that $\Gamma$ is torsion of unbounded exponent. Having
chosen $\gamma_1,\ldots,\gamma_{j-1}$, let
$\Gamma_{j-1}=\langle\gamma_1,\ldots,\gamma_{j-1}\rangle$ and choose
$\gamma_j$ with
\begin{equation}\label{eq:RS-relative-order}
 \ord(\gamma_j+\Gamma_{j-1})>4D_j.
\end{equation}
The subgroup $\Gamma_{j-1}$ is finite. The quotient
$\Gamma/\Gamma_{j-1}$ still has unbounded exponent: otherwise an integer annihilating the
quotient, multiplied by one annihilating $\Gamma_{j-1}$, would
annihilate $\Gamma$. Thus the choice is possible.
A nontrivial relation as in (D), at its last nonzero index $r$,
would give $k_r\gamma_r\in\Gamma_{r-1}$, contradicting
\eqref{eq:RS-relative-order}.
\end{proof}

\subsection{Exact integral factorization}
We identify a trigonometric polynomial
$W(t)=\sum_n a_ne^{int}$ with the Laurent polynomial
$\widetilde W(z)=\sum_n a_nz^n$ on the unit circle. Accordingly,
$W(\chi_{\gamma_j})$ denotes $\widetilde W\circ\chi_{\gamma_j}$.
By (D), expanding any finite product gives
\begin{equation}\label{eq:RS-factorization}
 \int_G\prod_jW_j(\chi_{\gamma_j})\dd m_G
 =\prod_j\int_\T W_j\dd m_\T
 \qquad(\deg W_j\le4D_j).
\end{equation}
Only the product of the constant coefficients can contribute the
trivial character. This exact identity also applies when the individual
characters have finite order and follows directly from
condition~\textup{(D)}.

Let $F_0=1$ and
\[
 F_L=\prod_{j=1}^Lf_j(\chi_{\gamma_j}),\qquad
 C=\prod_{j\ge1}(1+\delta_j)<\infty.
\]
Equations~\eqref{eq:RS-majorants} and \eqref{eq:RS-factorization} give
\begin{equation}\label{eq:RS-norm}
 \int F_L\dd m_G=1,\qquad
 \norm{F_L}_1\le\int\prod_{j\le L}U_j(\chi_{\gamma_j})\dd m_G
 =\prod_{j\le L}\int U_j\dd m_\T\le C.
\end{equation}

\subsection{The limiting measure and its exact spectrum}
Define
\begin{equation}\label{eq:RS-E}
 E=\left\{\sum_js_j\gamma_j:s_j\in S_j,
       \ s_j=0\text{ for all but finitely many }j\right\}.
\end{equation}
Representations in this formula are unique by (D), applied to the
difference of two representations. By (U), each fixed Fourier
coefficient of $F_L$ eventually stabilizes. Hence, for every
trigonometric polynomial $P$, the sequence $\int P F_L\dd m_G$ is
eventually constant, and by \eqref{eq:RS-norm}
\[
 \left|\int P F_L\dd m_G\right|
 \le C\norm{P}_\infty.
\]
The resulting functional on trigonometric polynomials therefore
extends, by uniform density, to a bounded functional on $C(G)$ and
hence to a measure $\mu\in M(G)$. Density of polynomials and
\eqref{eq:RS-norm} show that the entire sequence $F_Lm_G$ converges
to $\mu$ in $\sigma(M(G),C(G))$.
For $s_j\in S_j$ with only finitely many nonzero digits, the coefficients are precisely
\begin{equation}\label{eq:RS-mu-coeff}
 \widehat\mu\left(\sum_js_j\gamma_j\right)=\prod_jb_j(s_j),\qquad
 \widehat\mu(\gamma)=0\quad(\gamma\notin E).
\end{equation}
Thus $\mu(G)=1$, $\norm\mu\le C$ and
$\supp\widehat\mu=E$.

If $e+e'=\gamma$ with $e,e'\in E$, add their digit expansions.
At the last index $r$ where either has a nonzero digit, the sum
of the two digits is nonzero because $S_r\cap(-S_r)=\{0\}$.
It has absolute value at most $2N_r$, so (U) bounds $r$ in terms
of $\gamma$. Only finitely many digit choices remain. Hence
\begin{equation}\label{eq:RS-intersections}
 \#(E\cap(\gamma-E))<\infty\qquad(\gamma\in\Gamma).
\end{equation}
For $\gamma=0$, condition (D) excludes any such last index, giving
$E\cap(-E)=\{0\}$. The set $E$ is infinite, and if $0\ne e\in E$,
then $-e\notin E$; hence $E$ is proper. Proposition~\ref{prop:criterion}
proves the strongly small-2 property for the full spectral space $M_E(G)$.

\subsection{Singularity of the limiting measure}\label{subsec:RS-singularity}
We now apply the second-moment criterion from
Lemma~\ref{lem:second-moment}. Set $u_j=g_j(\chi_{\gamma_j})$.
Their required moments follow from \eqref{eq:RS-factorization}.

For $M\ge0$, define the positive tail measure
\[
 \rho_M=\mathop{\mathrm{w}^*\!-\!\lim}_{L\to\infty}
       \prod_{j=M+1}^Lp_j(\chi_{\gamma_j})m_G.
\]
Existence and mass one follow from (D), (U) and the same coefficient
stabilization argument: for every trigonometric polynomial $P$, the
integrals of $P$ against the finite tail products are eventually
constant and uniformly bounded by $\norm{P}_\infty$, so they define
$\rho_M$ by density and the Riesz representation theorem. For $j,k>M$,
$j\ne k$,
\begin{equation}\label{eq:RS-tail-moments}
 \int u_j\dd\rho_M=c,\qquad
 \int u_ju_k\dd\rho_M=c^2,
 \qquad \int u_j^2\dd\rho_M\le2a^2.
\end{equation}
For example, the first two identities involve the factors $g_jp_j$
and $g_kp_k$, each of degree at most $2N_j$ or $2N_k$; the last
bound follows from \eqref{eq:RS-positive}.
For $B_{M,L}=L^{-1}\sum_{j=M+1}^{M+L}u_j$ we obtain
\[
 \int|B_{M,L}|^2\dd m_G=c/L,\qquad
 \int|B_{M,L}-c|^2\dd\rho_M\le2a^2/L.
\]
Apply Lemma~\ref{lem:second-moment} to
$T_L=\sum_{j=M+1}^{M+L}u_j$ and $v_L=L$ under $m_G$ and $\rho_M$.
The normalized means are $0$ and $c$, and the variances are bounded
by $cL$ and $2a^2L$, respectively. Hence $\rho_M\perp m_G$.

Put $\nu_M=F_M\rho_M$. Multiplication by a fixed continuous
function is weak-$*$ continuous, and hence
$\rho_M=p_{M+1}(\chi_{\gamma_{M+1}})\rho_{M+1}$.
Consequently
\[
 \nu_{M+1}-\nu_M
 =iF_Mh_{M+1}(\chi_{\gamma_{M+1}})\rho_{M+1}.
\]
Polynomial majorants and \eqref{eq:RS-factorization} yield
\begin{align}
 \norm{\nu_{M+1}-\nu_M}
 &\le\int_G\left(\prod_{j\le M}U_j(\chi_{\gamma_j})\right)
             V_{M+1}(\chi_{\gamma_{M+1}})\dd\rho_{M+1}\notag\\
 &=\left(\prod_{j\le M}\int_\T U_j\dd m_\T\right)
             \int_\T V_{M+1}\dd m_\T
 \le C\delta_{M+1}.\label{eq:RS-tv}
\end{align}
The equality is first checked with a finite positive tail and then
passed to the weak-$*$ limit with its fixed continuous integrand.
Every $\nu_M$ is singular because $\nu_M\ll\rho_M\perp m_G$.
The increments in \eqref{eq:RS-tv} are summable, so
Lemma~\ref{lem:variation-limit} gives a singular variation-norm
limit.

Fix $\gamma$ and take $M\ge J(\gamma)$. By (U), no mixed product
$F_M\prod_{j=M+1}^L p_j(\chi_{\gamma_j})$ can acquire a contribution
to its coefficient at $\gamma$ from a later block. Its coefficient
therefore equals $\widehat F_M(\gamma)$, as does the stabilized
coefficient $\widehat\mu(\gamma)$. Passing to the limit in the
positive tail gives
$\widehat{\nu_M}(\gamma)=\widehat F_M(\gamma)=\widehat\mu(\gamma)$.
Fourier uniqueness identifies the variation-norm limit with $\mu$.
Therefore $\mu\perp m_G$.

\subsection{The convolution square}
By uniqueness of digit representations and \eqref{eq:RS-values},
\begin{equation}\label{eq:RS-fourth}
 \sum_{\gamma\in\Gamma}|\widehat\mu(\gamma)|^4
 =\prod_{j\ge1}\left(1+\frac{a^4}{N_j}\right)<\infty.
\end{equation}
The identity is obtained by monotone limits of sums of nonnegative
terms over finite words. The coefficients outside the countable set
$E$ vanish. Fact~\ref{lem:fourier-l2} gives
$\mu*\mu\in L^2(G)$, with squared norm equal to the product in
\eqref{eq:RS-fourth}.
This completes the proof of Theorem~\ref{thm:RS}.

For the circle, take $\gamma_0=1$ in \eqref{eq:RS-lambda}.
The construction also applies to every nontrivial compact connected
abelian group: the image of any nontrivial character is $\T$, so
that character has infinite order. The dual of the additive group
of $p$-adic integers is the Pr\"ufer group $\Z(p^\infty)$, whose
elements have arbitrarily large $p$-power orders. Thus the
unbounded-exponent construction covers every prime $p$, including
$p=2$.

\section{Bounded-exponent models with asymmetric spectra}\label{sec:finite}
The Rudin--Shapiro construction leaves only torsion duals of bounded
exponent. For the asymmetric case, the algebraic reduction in
Section~\ref{sec:algebra} requires two further models: products of
cyclic groups of order four and products of cyclic groups of odd
prime order. We construct them here. The quadratic Boolean function
used for the order-four model will also be needed in the Walsh
construction.

\subsection{Rectangular spectra and summable product perturbations}\label{sec:products}
We now use genuine product groups, so finite-coordinate integration
and independence are supplied by Lemma~\ref{lem:probability-product}.
We first formulate the spectral and measure-theoretic mechanisms
common to both bounded-exponent models.

\begin{corollary}[Asymmetric rectangles]\label{cor:rectangle}
Let $H_j$ be finite abelian groups. Suppose
$S_j\subseteq\widehat H_j$, $0\in S_j$, and
\begin{equation}\label{eq:localasym}
 S_j\cap(-S_j)=\{0\}\qquad(j\ge1).
\end{equation}
For
\begin{equation}\label{eq:rectangle}
 G=\prod_{j\ge1}H_j,\qquad
 E=\left\{\gamma\in\bigoplus_{j\ge1}\widehat H_j:
                  \gamma_j\in S_j\text{ for all }j\right\},
\end{equation}
the set $E$ is strongly small-2 and $E\cap(-E)=\{0\}$.
\end{corollary}
\begin{proof}
Fix $\gamma\in\widehat G$ and let $J=\{j:\gamma_j\ne0\}$, a finite
set. If $\xi\in E\cap(\gamma-E)$ and $j\notin J$, then both
$\xi_j$ and $-\xi_j$ belong to $S_j$. Thus $\xi_j=0$. The remaining
coordinates lie in finitely many finite sets, so
$E\cap(\gamma-E)$ is finite. Apply Proposition~\ref{prop:criterion}.
The last assertion follows directly from \eqref{eq:localasym}.
\end{proof}

\begin{lemma}[Product perturbation]\label{lem:product}
Let $H_j$ be finite abelian groups with normalized Haar measures $m_j$.
Let $p_j,f_j:H_j\to\C$ satisfy
\begin{equation}\label{eq:producthyp}
 p_j>0,\quad \int p_j\dd m_j=\int f_j\dd m_j=1,
 \quad\delta_j:=\norm{f_j-p_j}_1,
 \quad\sum_j\delta_j<\infty.
\end{equation}
On $G=\prod_jH_j$ put $\rho=\bigotimes_j(p_jm_j)$, as in
Lemma~\ref{lem:probability-product}.
There exists a unique measure $\mu$ whose marginal on the first $N$
blocks has density $\prod_{j=1}^N f_j$. It satisfies
\begin{equation}\label{eq:productnorm}
 \mu(G)=1,\qquad \mu\ll\rho,\qquad
 \norm\mu=\prod_{j=1}^{\infty}\norm{f_j}_1<\infty.
\end{equation}
If every $f_j$ is everywhere nonzero, then $|\mu|\sim\rho$.
Writing $C=\prod_j\norm{f_j}_1$ and
\[
 \nu_N=\bigotimes_{j\le N}(f_jm_j)\ \otimes\!\bigotimes_{j>N}(p_jm_j),
\]
we have the quantitative estimate
\begin{equation}\label{eq:tvbound}
 \norm{\mu-\nu_N}\le C\sum_{j>N}\delta_j.
\end{equation}
\end{lemma}
\begin{proof}
On the probability space $(G,\rho)$ set
$r_j=f_j/p_j$ and $R_N=\prod_{j=1}^N r_j$. Independence gives
\[
 \norm{R_N}_{L^1(\rho)}=\prod_{j\le N}\norm{f_j}_1,
 \qquad
 \norm{R_{N+1}-R_N}_{L^1(\rho)}
   =\left(\prod_{j\le N}\norm{f_j}_1\right)\delta_{N+1}.
\]
Since $1\le\norm{f_j}_1\le1+\delta_j$, the product $C$ is finite.
The series of these increment norms converges, so $R_N$ has a limit
$R$ in $L^1(\rho)$. Define $\mu=R\rho$. Then
$\nu_N=R_N\rho\to\mu$ in variation norm, with \eqref{eq:tvbound}.
The finite marginals stabilize, since all the factors have mass one,
and therefore give the asserted marginal densities. Also
\[
 \norm{\nu_N}=\norm{R_N}_{L^1(\rho)}
    =\prod_{j\le N}\norm{f_j}_1\longrightarrow C.
\]
Since $|\norm{\nu_N}-\norm\mu|\le\norm{\nu_N-\mu}$, variation
convergence yields $\norm\mu=C$; the mass formula follows in the same
way (or directly from the stabilized marginals). Finite-coordinate
continuous functions are uniformly dense in $C(G)$, so the marginals
determine $\mu$ uniquely.

Finally, monotone convergence, as in \eqref{eq:tonelli-summability}, gives
\[
 \int_G\sum_j|r_j-1|\dd\rho=\sum_j\delta_j<\infty.
\]
Thus $\sum_j|r_j(x_j)-1|<\infty$ for $\rho$-almost every $x$.
If all factors are nonzero, the elementary convergence criterion for
infinite products implies that $\prod_jr_j(x_j)$ converges to a
nonzero limit. Indeed, eventually $|r_j-1|<1/2$, and the series of
the corresponding principal logarithms converges absolutely.
An almost-everywhere convergent subsequence of $R_N\to R$ in $L^1$
identifies that product with $R$. Hence $|R|>0$ almost everywhere
and $|\mu|\sim\rho$.
\end{proof}

\begin{theorem}[General construction scheme]\label{thm:scheme}
Assume \eqref{eq:producthyp}, $\rho\perp m$, and
\[
 S_j:=\supp\widehat f_j,\qquad S_j\cap(-S_j)=\{0\}.
\]
Then the rectangle $E$ from \eqref{eq:rectangle} is strongly small-2
and not Riesz. The measure of Lemma~\ref{lem:product} satisfies
\begin{equation}\label{eq:fullsupport}
 \mu\perp m,\quad\mu(G)=1,\quad\supp\widehat\mu=E.
\end{equation}
If in addition
\begin{equation}\label{eq:fourthhyp}
 \sum_j b_j<\infty,\qquad
 b_j=\sum_{\eta\ne0}|\widehat f_j(\eta)|^4,
\end{equation}
then
\begin{equation}\label{eq:squarenorm}
 \mu*\mu\in L^2(G),\qquad
 \norm{\mu*\mu}_2^2=\prod_j(1+b_j).
\end{equation}
\end{theorem}
\begin{proof}
The marginal identities imply, for each character of finite block support,
\begin{equation}\label{eq:fourierproduct}
 \widehat\mu(\gamma)
   =\prod_{j:\gamma_j\ne0}\widehat f_j(\gamma_j).
\end{equation}
As $\widehat f_j(0)=1$, this proves the exact support assertion.
The measure has mass one and satisfies $\mu\ll\rho\perp m$.
Corollary~\ref{cor:rectangle} gives strong small-2, and the nonzero
singular measure excludes the Riesz property.

Since $\widehat G=\bigoplus_j\widehat H_j$ is countable,
Fact~\ref{lem:fourier-l2} reduces \eqref{eq:squarenorm} to the
following identity, obtained by summing first over characters
supported in the first $N$ blocks:
\[
 \sum_{\gamma\in\widehat G}|\widehat\mu(\gamma)|^4
   =\lim_{N\to\infty}\prod_{j=1}^N(1+b_j)<\infty.\qedhere
\]
\end{proof}

\subsection{The order-four model}\label{subsec:boolean-finite}
For $k\ge1$, let $q=2^k$, $V=\F_2^k\times\F_2^k$, and
\[
 b(u,v)=(-1)^{u\cdot v}.
\]
Identifying $V$ with its dual by the standard dot products, summation
first in the second variable gives
\begin{align}
 \widehat b(s,t)
 &=q^{-2}\sum_{u,v\in\F_2^k}(-1)^{u\cdot v+s\cdot u+t\cdot v}
 \notag\\
 &=q^{-1}(-1)^{s\cdot t}=q^{-1}b(s,t).
 \label{eq:bent}
\end{align}

\begin{lemma}\label{lem:boolean}
Let $H=(\Z/4\Z)\times V$. Fix
$0<a\le1/2$ and define
\begin{align}
 g(t,w)&=ab(w)\cos(\pi t/2),\notag\\
 h(t,w)&=aq\ind_{\{0\}}(w)\sin(\pi t/2),
 \qquad p=1+g,\quad f=p+ih.\label{eq:booleanfactor}
\end{align}
Then $p\ge1-a>0$, both $p$ and $f$ have integral one, and
\begin{equation}\label{eq:booleancost}
 \norm{f-p}_1=\frac{a}{2q},\qquad
 1\le\norm f_1\le1+\frac{a}{2q}.
\end{equation}
Viewing the signs $\pm1$ as their residue classes in $\Z/4\Z$,
the exact Fourier support of $f$ is
\begin{equation}\label{eq:booleansupport}
 S_q=\{(0,0)\}\cup\{(b(z),z):z\in V\},
 \qquad S_q\cap(-S_q)=\{0\},
\end{equation}
and
\begin{align}
 \widehat f(b(z),z)&=\frac aq b(z)\quad(z\in V),
 \label{eq:booleanvalues}\\
 \E_{m_H}g&=0,\qquad \E_{m_H}g^2=a^2/2,\qquad |g|\le a,
 \label{eq:booleanmoments}\\
 \norm f_2^2&=1+a^2,\qquad
 \sum_{\eta\ne0}|\widehat f(\eta)|^4=a^4/q^2.
 \label{eq:booleanparseval}
\end{align}
In particular, $f$ has no zeros.
\end{lemma}
\begin{proof}
On $\Z/4\Z$, the cosine has mean zero and mean square $1/2$,
while the sine has mean zero and mean absolute value $1/2$.
Since $b^2=1$, this proves the mass and moment statements.
Also $m_V(\{0\})=q^{-2}$, so $\norm h_1=a/(2q)$. Positivity of $p$ and
$|p+ih|\le p+|h|$ prove \eqref{eq:booleancost}.

The sine coefficients at $1,-1$ are $1/(2i),-1/(2i)$.
Together with \eqref{eq:bent} and
$\widehat{\ind_{\{0\}}}(z)=q^{-2}$, this gives
\[
 \widehat f(1,z)=\frac a{2q}(b(z)+1),\qquad
 \widehat f(-1,z)=\frac a{2q}(b(z)-1).
\]
The coefficients at $(0,z)$ equal $\ind_{\{0\}}(z)$, and all other
coefficients vanish. Exactly one frequency in each pair survives.
Since $-z=z$ and $1\ne-1$ in the carrier group, its negative does
not survive. This includes $z=0$, whose selected nonconstant
frequency is $(1,0)$. There are $q^2$ nonconstant coefficients of
modulus $a/q$, proving \eqref{eq:booleanparseval} by Parseval.
Finally $\operatorname{Re}f=p>0$.
\end{proof}

\begin{proposition}[Order-four product model]\label{prop:carrier-products}
There are a set and a measure satisfying \eqref{eq:main} and
\eqref{eq:global-fi} on $(\Z/4\Z)^{\N}$.
\end{proposition}
\begin{proof}
Fix $0<a\le1/2$ and take
\[
 q_j=2^j,\qquad H_j=(\Z/4\Z)\times\F_2^{2j},
 \qquad \rho=\bigotimes_j(p_jm_j).
\]
Under Haar measure $g_j$ has mean zero, whereas under $\rho$
its mean is $\int g_j(1+g_j)\dd m_j=a^2/2$.
Corollary~\ref{lem:statistic} gives $\rho\perp m$. The perturbation
costs and the fourth-power sums are summable, so
Theorem~\ref{thm:scheme} applies. It gives, more explicitly,
\begin{equation}\label{eq:carrier-bounds}
 1\le\norm\mu\le e^{a/2},\qquad
 \norm{\mu*\mu}_2^2=\prod_{j\ge1}\left(1+\frac{a^4}{4^j}\right).
\end{equation}
The local supports omit $(-1,0)$, so the resulting spectrum is proper.

The group $\prod_jH_j$ is a quotient of $(\Z/4\Z)^{\N}$: in disjoint successive blocks of coordinates,
retain one coordinate modulo four and reduce the following $2j$
coordinates modulo two. Apply Lemma~\ref{lem:lift}. Finite intersections
remain finite when a set is viewed in a larger discrete dual group:
for a shift outside the subgroup carrying the set the intersection
is empty. Thus all the asserted properties are preserved.
\end{proof}

\begin{remark}\label{rem:blocks}
The blocks $H_j$ grow in size. After reindexing into cyclic coordinates,
the spectrum need not remain a rectangle in those smaller coordinates:
the selected sign in \eqref{eq:booleansupport} depends on the entire
Boolean block. The order-four carrier supplies two distinct
frequencies $1$ and $-1$; this is not a construction on a group of
exponent two.
\end{remark}

\subsection{Elementary abelian groups of odd exponent}\label{subsec:odd}

We next treat groups of odd prime exponent. Throughout this subsection,
$p$ denotes an odd prime, not a density.

\begin{lemma}[An elementary quadratic Fourier calculation]\label{lem:gauss}
Let $k\ge1$, let $K=\F_{p^{2k}}$, $Q=|K|$, and let
$\eta:K\to\{0,1,-1\}$ be the
quadratic character, with $\eta(0)=0$. Choose a nontrivial additive
character $\psi$ of $K$, and identify its additive dual by
$z\mapsto(w\mapsto\psi(zw))$. There is a real function $b$ on $K$
such that
\begin{equation}\label{eq:gaussprops}
 |b|\le1,\quad b^2=\ind_{K\setminus\{0\}},\quad
 \widehat b(0)=0,\quad
 \widehat b(z)=\frac{\eta(z)}{\sqrt Q}\quad(z\ne0).
\end{equation}
Moreover $\eta(-z)=\eta(z)$.
\end{lemma}
\begin{proof}
The multiplicative group of a finite field is cyclic, and
$Q\equiv1\pmod4$. Hence $-1$ is a square and $\eta(-1)=1$.
Put
\[
 \tau=\sum_{w\in K}\eta(w)\overline{\psi(w)}.
\]
Changing $w$ to $-w$ shows that $\tau$ is real. For its squared
modulus, terms with a zero variable vanish, and the substitution
$x=ty$ gives
\begin{align*}
 |\tau|^2
 &=\sum_{t\in K^\times}\eta(t)
       \sum_{y\in K^\times}\overline{\psi((t-1)y)}\\
 &=(Q-1)-\sum_{\substack{t\in K^\times\\t\ne1}}\eta(t)=Q.
\end{align*}
Here the inner sum is $Q-1$ at $t=1$ and $-1$ otherwise, and
$\sum_{t\in K^\times}\eta(t)=0$.
For $z\ne0$, changing variables from $w$ to $zw$ yields
\[
 \widehat\eta(z)=\frac{\tau}{Q}\eta(z).
\]
Thus $b=(\tau/\sqrt Q)\eta$ has all the stated properties, since
$\tau^2=Q$.
\end{proof}

\begin{lemma}[An odd-characteristic factor]\label{lem:odd}
On $H=(\Z/p\Z)\times K$, with normalized Haar measure $m_H$, fix
$0<a\le1/2$, use the function $b$ from
Lemma~\ref{lem:gauss}, and define
\begin{align}
 g(t,w)&=a b(w)\cos(2\pi t/p),\notag\\
 h(t,w)&=a\sqrt Q\left(\ind_{\{0\}}(w)-\frac1Q\right)
                         \sin(2\pi t/p),\label{eq:oddfactor}\\
 P&=1+g,\qquad f=P+ih.\notag
\end{align}
Then $P\ge1-a>0$, both $P$ and $f$ have integral one, and
\begin{equation}\label{eq:oddcost}
 \norm{f-P}_1\le\frac{2a}{\sqrt Q}.
\end{equation}
Furthermore, viewing $\pm1$ as residue classes in $\Z/p\Z$,
\begin{equation}\label{eq:oddsupport}
 S:=\supp\widehat f
   =\{(0,0)\}\cup\{(\eta(z),z):z\in K^\times\},
 \qquad S\cap(-S)=\{0\},
\end{equation}
with
\begin{equation}\label{eq:oddvalues}
 \widehat f(\eta(z),z)=\frac a{\sqrt Q}\eta(z)\quad(z\ne0).
\end{equation}
Also
\begin{equation}\label{eq:oddmoments}
 \E_{m_H} g=0,\qquad
 \E_{m_H} g^2=\frac{a^2}{2}\left(1-\frac1Q\right),
 \qquad
 \sum_{\zeta\ne0}|\widehat f(\zeta)|^4
     =\frac{(Q-1)a^4}{Q^2}.
\end{equation}
\end{lemma}
\begin{proof}
For odd $p$, the mean of the cosine on $\Z/p\Z$ is zero, and its
mean square is $1/2$. Together with \eqref{eq:gaussprops}, this proves
the assertions about $P$ and the first two identities in
\eqref{eq:oddmoments}. The sine has mean zero, so $\int f=1$.
Also
\[
 \E_K\left|\ind_{\{0\}}-Q^{-1}\right|
    =\frac2Q(1-Q^{-1}),
\]
which proves \eqref{eq:oddcost}.

The Fourier transform of $\ind_{\{0\}}-Q^{-1}$ is zero at the
zero character and $Q^{-1}$ at every other character. Hence for
$z\ne0$,
\begin{align*}
 \widehat f(1,z)&=\frac a{2\sqrt Q}(\eta(z)+1),\\
 \widehat f(-1,z)&=\frac a{2\sqrt Q}(\eta(z)-1).
\end{align*}
Both coefficients at $(1,0)$ and $(-1,0)$ vanish. Outside the constant
frequency and the two families $(1,z)$ and $(-1,z)$ with $z\ne0$,
all coefficients vanish.
The subtraction of $Q^{-1}$ in \eqref{eq:oddfactor} is essential
for precisely this cancellation at $z=0$.
For $z\ne0$ exactly one of the displayed coefficients survives.
Since $\eta(-z)=\eta(z)$ and $1\ne-1$ in $\Z/p\Z$, its negative
frequency does not survive. This proves \eqref{eq:oddsupport} and
\eqref{eq:oddvalues}. There are $Q-1$ nonconstant coefficients of
modulus $a/\sqrt Q$, proving the last identity in
\eqref{eq:oddmoments}.
\end{proof}

\begin{proposition}[Odd-prime product model]\label{prop:odd-products}
For every odd prime $p$, there are a set $E$ and a measure $\mu$ on
$(\Z/p\Z)^{\N}$ satisfying \eqref{eq:main} and
\eqref{eq:global-fi}.
\end{proposition}
\begin{proof}
Use Lemma~\ref{lem:odd} on
\[
 K_j=\F_{p^{2j}},\quad Q_j=p^{2j},\quad
 H_j=(\Z/p\Z)\times K_j,
\]
with a fixed $0<a\le1/2$. Put $\rho=\bigotimes_j(P_jm_j)$.
Under Haar measure $g_j$ has mean zero. Under $\rho$ it has mean
\[
 \E_\rho g_j=\E_{m_j}g_j^2
   =\frac{a^2}{2}(1-Q_j^{-1})\longrightarrow a^2/2.
\]
The $g_j$ are uniformly bounded, so Corollary~\ref{lem:statistic} gives
$\rho\perp m$.
The perturbation costs are summable since $\sum_jQ_j^{-1/2}<\infty$.
Theorem~\ref{thm:scheme} and \eqref{eq:oddsupport} now give the strongly
small-2 set and the singular measure of full Fourier support. Moreover,
\begin{equation}\label{eq:oddnormsquare}
 \norm{\mu*\mu}_2^2
   =\prod_{j\ge1}\left(1+\frac{(Q_j-1)a^4}{Q_j^2}\right)<\infty.
\end{equation}
Each block is additively isomorphic to $(\Z/p\Z)^{2j+1}$; hence the
countable product is isomorphic to $(\Z/p\Z)^{\N}$.
The local sets omit every nonzero frequency of the form $(0,z)$,
so the resulting set $E$ is proper.
\end{proof}

\section{A nonrectangular construction on the Walsh group}\label{sec:walsh}
Let
\[
 W=(\Z/2\Z)^{\N},\qquad
 \Lambda=\widehat W=\bigoplus_{\N}\F_2.
\]
On this group the asymmetric rectangle criterion cannot produce an
infinite spectrum. In fact, a rectangle with zero in every local
support and infinitely many nontrivial local supports cannot be
small-2. Choosing one nonzero character in each nontrivial block
gives an infinite subgroup $D$ contained in the rectangle. The
normalized Haar measure on $D^\perp$, viewed as a measure on $W$,
is singular because $D^\perp$ has infinite index, and its Fourier
transform is $\ind_D$. It is therefore a singular convolution
idempotent supported spectrally in the rectangle, which rules out
small-2. Moreover, Boolean characters are real, so adding a purely
imaginary perturbation to a real function cannot cancel any of its
nonzero Fourier coefficients.

We instead allow the zero shift as the sole exception, and construct
the spectrum by adding layers with no new intersections at old
nonzero shifts. The criterion needed for this step is proved below;
the subsequent recursion uses real, frequency-dependent perturbations.

\begin{theorem}[Walsh construction]\label{thm:walsh}
For every $\varepsilon>0$ there are a countably infinite proper
set $E\subsetneq\Lambda$ and a real signed measure $\mu\in M(W)$
such that
\begin{gather}
 \#(E\cap(E+d))<\infty\quad(d\ne0),\label{eq:walsh-goal-thin}\\
 \mu(W)=1,\qquad \mu\perp m,\qquad
 \supp\widehat\mu=E,\qquad\norm\mu\le1+\varepsilon,
 \label{eq:walsh-goal-measure}\\
 \mu*\mu\in L^2(W),\qquad
 \norm{\mu*\mu-m}_2^2\le\varepsilon.\label{eq:walsh-goal-square}
\end{gather}
In particular, $E$ is strongly small-2 and is not Riesz.
\end{theorem}
The measure $\mu$ need not be positive: having mass one and variation
norm close to one does not make it a probability measure.

\subsection{The spectral criterion: one exceptional shift}
The next criterion concerns intersections with translates, with the
zero shift excluded. It is a specialization of Pigno's criterion
\cite[Theorem~2]{Pigno}: in the compact case the dual is discrete, we take
the exceptional Riesz set to be $\{0\}$, and all remaining intersections
are finite, hence Riesz. We give a direct proof for compact metrizable
abelian groups; the application to the Walsh group replaces
Proposition~\ref{prop:criterion}.

\begin{proposition}[One exceptional shift]\label{prop:walsh-criterion}
Let $K$ be a compact metrizable abelian group, let
$\Lambda=\widehat K$, and suppose that $E\subseteq\Lambda$ satisfies
\begin{equation}\label{eq:walsh-thin}
 \#(E\cap(E+d))<\infty\qquad(d\in\Lambda\setminus\{0\}).
\end{equation}
Then $E$ is strongly small-2.
\end{proposition}
\begin{proof}
Take $\alpha,\beta\in M_E(K)$ and set $\rho=|\alpha|*|\beta|$.
Choose a Borel Haar-null set $N$ carrying the singular part of $\rho$.
On $K\times K$, define
\begin{equation}\label{eq:walsh-lambda}
 \dd\lambda(x,y)=\ind_N(x+y)\dd\alpha(x)\dd\beta(y).
\end{equation}
Write $\alpha=\alpha_a+\alpha_s$ and
$\beta=\beta_a+\beta_s$ for their Lebesgue decompositions.
Fubini's theorem and translation invariance of Haar measure show that
the restriction to $\{(x,y):x+y\in N\}$ of any product measure with
at least one absolutely continuous factor is zero; the same holds for
its variation. Hence
\begin{equation}\label{eq:walsh-sing-product}
 \lambda=\ind_N(x+y)\,\alpha_s\otimes\beta_s,
 \qquad |\lambda|\le |\alpha_s|\otimes|\beta_s|.
\end{equation}
Note that neither $\alpha_s$ nor $\beta_s$ is asserted to belong to
$M_E(K)$.

If $u\ne v$ in $\Lambda$, put $d=u-v\ne0$.
The measure
\[
 C_d=\alpha*(\chi_d\beta)
\]
has Fourier support in $E\cap(d+E)$, so it has a polynomial density
by Fact~\ref{lem:finitefourier}. The identity
\[
 \overline{\chi_u(x+y)}\chi_{u-v}(y)
 =\overline{\chi_u(x)}\,\overline{\chi_v(y)}
\]
therefore gives
\[
 \widehat\lambda(u,v)=\int_N\overline{\chi_u}\dd C_d=0.
\]
Thus the Fourier support of $\lambda$ lies in the diagonal.

For fixed $v\in\Lambda$, define a twisted first marginal by
\[
 \theta_v(A)=\int_{A\times K}\overline{\chi_v(y)}\dd\lambda(x,y).
\]
By \eqref{eq:walsh-sing-product},
$|\theta_v|\le\norm{\beta_s}|\alpha_s|$, so $\theta_v$ is singular.
On the other hand,
$\widehat{\theta_v}(u)=\widehat\lambda(u,v)=0$ for $u\ne v$.
Fourier uniqueness yields
$\theta_v=\widehat\lambda(v,v)\chi_vm_K$. Consequently
$\theta_v=0$. The diagonal coefficients also vanish, and
$\lambda=0$. Finally, by the definition of the total variation of a
restricted product measure,
\[
 \norm\lambda
 =\int_{K^2}\ind_N(x+y)\dd|\alpha|(x)\dd|\beta|(y)
 =\rho(N).
\]
It follows that $\rho(N)=0$ and hence $\rho\ll m_K$.
\end{proof}

\subsection{Disjoint spectra close to one common function in \texorpdfstring{$L^1$}{L1}}
Retain $q=2^k$, $V=\F_2^k\times\F_2^k$ and the real sign function
$b$ from \eqref{eq:bent}. Then
\[
 \widehat b(z)=q^{-1}b(z),\qquad b(0)=1,\qquad
 \sum_{z\in V}b(z)=q.
\]
For $s\in\{-1,1\}$, define the real function
\begin{equation}\label{eq:walsh-ps}
 p_s=b+sq\ind_{\{0\}}.
\end{equation}
Here $\ind_{\{0\}}$ is a function on $V$, not a Dirac measure;
its integral is $q^{-2}$. Since its Fourier transform is constant
and equal to $q^{-2}$,
\begin{equation}\label{eq:walsh-pshat}
 \widehat p_s(z)=\frac{b(z)+s}{q}
 =\begin{cases}2s/q,&b(z)=s,\\0,&b(z)=-s.\end{cases}
\end{equation}
Thus the supports of $\widehat p_+$ and $\widehat p_-$ are disjoint,
yet both functions are close to the same $b$:
\begin{equation}\label{eq:walsh-pclose}
 \norm{p_s-b}_1=q^{-1},\qquad\norm{p_s}_1\le1+q^{-1}.
\end{equation}
These functions are not nonnegative and will not be used as probability densities.

To obtain many copies, let $r\ge1$ and put
$H=\F_2\times V^r$. With variables $(t,z_1,\ldots,z_r)$, set
\begin{equation}\label{eq:walsh-hsigma}
 h(t,z)=(-1)^t\prod_{\ell=1}^rb(z_\ell),\qquad
 h_\sigma(t,z)=(-1)^t\prod_{\ell=1}^rp_{\sigma_\ell}(z_\ell)
 \quad(\sigma\in\{-1,1\}^r).
\end{equation}
The function $h$ takes the values $\pm1$ with equal probabilities.
All the functions $h_\sigma$ are real and have mean zero, because
of the extra factor $(-1)^t$. Their exact spectra are
\begin{equation}\label{eq:walsh-Tsigma}
 T_\sigma=\{(1,z_1,\ldots,z_r):b(z_\ell)=\sigma_\ell\ (1\le\ell\le r)\}.
\end{equation}
They are pairwise disjoint and omit zero. Every coefficient on
$T_\sigma$ is nonzero and equals
$(2/q)^r\prod_{\ell=1}^r\sigma_\ell$.

Expanding a telescoping difference of products and using
\eqref{eq:walsh-pclose}, together with the product Haar measure,
gives
\begin{equation}\label{eq:walsh-d}
 \norm{h_\sigma-h}_1
 \le\sum_{i=1}^r q^{-1}(1+q^{-1})^{i-1}
 =(1+q^{-1})^r-1=:d(r,q).
\end{equation}
Also $\#\{z:b(z)=s\}=q(q+s)/2$, and hence
\begin{align}
 \sum_{\eta\in\widehat H}|\widehat h_\sigma(\eta)|^4
 &=\prod_{\ell=1}^r\frac{8(q+\sigma_\ell)}{q^3}\notag\\
 &\le\left(\frac{8(q+1)}{q^3}\right)^r=:v(r,q).
 \label{eq:walsh-v}
\end{align}
For fixed $r$, both $d(r,q)$ and $v(r,q)$ tend to zero as $q$
tends to infinity through powers of two.

\begin{lemma}[Spectral splitting]\label{lem:walsh-splitting}
Given a positive integer $L$ and numbers $\delta,\zeta>0$, there
are a finite Boolean group $H$, a mean-zero sign function
$h:H\to\{-1,1\}$, and $L$ real mean-zero functions $h_1,\ldots,h_L$
with pairwise disjoint, nonempty Fourier supports such that
\[
 \norm{h_i-h}_1\le\delta,\qquad
 \sum_\eta|\widehat h_i(\eta)|^4\le\zeta\quad(1\le i\le L).
\]
\end{lemma}
\begin{proof}
Choose $r\ge1$ with $2^r\ge L$, choose $L$ distinct labels $\sigma$,
and take a sufficiently large power of two $q$ in
\eqref{eq:walsh-d} and \eqref{eq:walsh-v}.
\end{proof}

\begin{remark}
Disjoint Fourier supports imply $L^2$ orthogonality, but do not
prevent closeness in $L^1$. The functions $h_\sigma$ are not
uniformly bounded: the spectral splitting is paid for by high
values on rare atoms. The common function $h$ remains bounded by
one.
\end{remark}

\subsection{A consistent frequency-dependent recursion}
Fix $a=1/2$, let $\varepsilon_j=\varepsilon2^{-j}$, and start with
$F_0=1$. We construct finite Boolean groups $H_j$ and real
polynomials $F_j$ on $G_j=H_1\times\cdots\times H_j$. Write
$\Lambda_j=\widehat G_j$, with $\Lambda_0=\{0\}$. We identify each
$\Lambda_j$ with $\Lambda_j\times\{0\}\subset\Lambda_{j+1}$ and with
the corresponding subgroup of $\Lambda$ by appending zero coordinates.
Suppose that
\[
 F_{j-1}(x)=\sum_{\xi\in E_{j-1}}c_\xi\chi_\xi(x),
 \qquad E_{j-1}=\supp\widehat F_{j-1},
\]
has already been chosen. Put
\[
 A_{j-1}=\sum_{\xi\in E_{j-1}}|c_\xi|,\qquad
 B_{j-1}=\sum_{\xi\in E_{j-1}}|c_\xi|^4.
\]
These are finite positive numbers; the coefficient at zero will
always be one. Choose the least $r_j\ge1$ such that
$2^{r_j}\ge\#E_{j-1}$, and any injection
\[
 \xi\longmapsto\sigma(\xi)\quad\hbox{from }E_{j-1}
       \hbox{ into }\{-1,1\}^{r_j}.
\]
Choose a power of two $q_j\ge2$ sufficiently large that
\begin{equation}\label{eq:walsh-budgets}
 aA_{j-1}d(r_j,q_j)\le\varepsilon_j,\qquad
 a^4B_{j-1}v(r_j,q_j)\le\varepsilon_j.
\end{equation}
The limits after \eqref{eq:walsh-v} ensure existence.

On the new block $H_j=\F_2\times(\F_2^{k_j}\times\F_2^{k_j})^{r_j}$,
where $q_j=2^{k_j}$, use the functions $h_j,h_{j,\sigma}$ from
\eqref{eq:walsh-hsigma} and define
\begin{equation}\label{eq:walsh-recursion}
 \begin{split}
 P_j(y)&=1+ah_j(y),\\
 F_j(x,y)&=F_{j-1}(x)+a\sum_{\xi\in E_{j-1}}
        c_\xi\chi_\xi(x)h_{j,\sigma(\xi)}(y).
 \end{split}
\end{equation}
All coefficients and functions are real. Because the $h_{j,\sigma}$
have mean zero, the marginal identities are exact:
\begin{equation}\label{eq:walsh-consistency}
 \int_{H_j}F_j(x,y)\dd m_{H_j}(y)=F_{j-1}(x),\qquad
 \int_{G_j}F_j\dd m_{G_j}=1.
\end{equation}
The error relative to multiplication by the common positive factor
is
\begin{equation}\label{eq:walsh-error}
 R_j:=F_j-F_{j-1}P_j
 =a\sum_{\xi\in E_{j-1}}c_\xi\chi_\xi
       (h_{j,\sigma(\xi)}-h_j).
\end{equation}
By \eqref{eq:walsh-d} and \eqref{eq:walsh-budgets},
\begin{equation}\label{eq:walsh-error-bound}
 \norm{R_j}_{L^1(G_j)}\le\varepsilon_j.
\end{equation}
Since $P_j\ge1-a>0$, its integral is one, and it depends only on
the new coordinate block,
\begin{equation}\label{eq:walsh-L1}
 \norm{F_j}_1\le\norm{F_{j-1}P_j}_1+\norm{R_j}_1
 =\norm{F_{j-1}}_1+\norm{R_j}_1
 \le1+\sum_{i\le j}\varepsilon_i.
\end{equation}
The label $\sigma(\xi)$ depends on the old frequency. Replacing all
these labels by one common label would destroy the spectral
property needed below.

\subsection{A singular reference measure and convergence in variation}
The countable product $G=\prod_{j\ge1}H_j$ is topologically
isomorphic to $W$, since every block is a nontrivial finite Boolean
group. We make this identification. Consider the positive
probability product
\begin{equation}\label{eq:walsh-rho}
 \rho=\bigotimes_{j\ge1}(P_jm_{H_j}).
\end{equation}
Under $m$, the $h_j$ are independent mean-zero signs. Under $\rho$,
Lemma~\ref{lem:probability-product} again gives independence, and
\[
 \E_\rho h_j=\int h_j(1+ah_j)\dd m_{H_j}=a.
\]
The bound $|h_j|=1$ and Corollary~\ref{lem:statistic} therefore imply
$\rho\perp m$. This is a genuine product-measure application, unlike
the moment-factorization argument in Section~\ref{sec:RS}.

Set $Q_0=1$, $Q_j=\prod_{i=1}^jP_i$, and $Z_j=F_j/Q_j$.
The denominators are positive. Equation~\eqref{eq:walsh-error}
gives
\begin{equation}\label{eq:walsh-change-measure}
 Z_j-Z_{j-1}=\frac{R_j}{Q_j},\qquad
 \norm{Z_j-Z_{j-1}}_{L^1(\rho)}=\norm{R_j}_{L^1(G_j)}
 \le\varepsilon_j.
\end{equation}
The equality of norms follows because the $G_j$ marginal of $\rho$
has density $Q_j$. Thus $Z_j$ converges in $L^1(\rho)$ to a real
function $Z$. Moreover, telescoping \eqref{eq:walsh-change-measure}
gives
\[
 \norm{Z-1}_{L^1(\rho)}
 \le \sum_{j\ge1}\norm{Z_j-Z_{j-1}}_{L^1(\rho)}
 \le \sum_{j\ge1}\varepsilon_j
 =\varepsilon.
\]
Define
\begin{equation}\label{eq:walsh-mu}
 \mu=Z\rho.
\end{equation}
Then $\mu\perp m$ and
\begin{equation}\label{eq:walsh-norm-close}
 \norm\mu\le1+\varepsilon,\qquad\norm{\mu-\rho}\le\varepsilon.
\end{equation}
Also $\int Z_j\dd\rho=\int F_j\dd m=1$, so $\mu(W)=1$.

For a function $\psi$ depending on the first $n$ blocks and $j\ge n$,
\eqref{eq:walsh-consistency} implies
\[
 \int\psi Z_j\dd\rho=\int_{G_j}\psi F_j\dd m_{G_j}
 =\int_{G_n}\psi F_n\dd m_{G_n}.
\]
Passing to the $L^1(\rho)$ limit proves that the $G_n$ marginal of
$\mu$ has density exactly $F_n$.

\begin{remark}[The two modes of convergence]
The measures $Z_j\rho$ converge to $\mu$ in variation norm.
The absolutely continuous measures $F_jm$ converge only weak-$*$,
as follows from their consistent marginals, uniform norm bound,
and density of finite-coordinate continuous functions. Singularity
is justified by $\mu\ll\rho\perp m$, not by weak-$*$ convergence alone.
\end{remark}

\subsection{Exact support and absence of late overlaps}
Let $T_{j,\sigma}=\supp\widehat h_{j,\sigma}$, and define the new
layer
\[
 L_j=\{(\xi,\eta):\xi\in E_{j-1},\ \eta\in T_{j,\sigma(\xi)}\}
 \subset\Lambda_{j-1}\oplus\widehat H_j.
\]
As $0\notin T_{j,\sigma}$, the exact spectral recursion is
\begin{equation}\label{eq:walsh-support}
 E_j=(E_{j-1}\times\{0\})\,\dot\cup\,L_j.
\end{equation}
Old coefficients remain unchanged, and each new coefficient equals
\begin{equation}\label{eq:walsh-new-coeff}
 \widehat F_j(\xi,\eta)
 =ac_\xi\widehat h_{j,\sigma(\xi)}(\eta)\ne0.
\end{equation}
There is no addition of different contributions to the same
coefficient. The exact marginal identities therefore give
\begin{equation}\label{eq:walsh-full-E}
 E=\bigcup_{j\ge0}E_j=\supp\widehat\mu.
\end{equation}
Each layer is nonempty, so $E$ is infinite. It is proper: a nonzero
frequency in the first block with zero in its extra $t$ coordinate
is absent from $E_1$, and no later layer can introduce a frequency
supported only in that first block.

For every $d\in\Lambda_{j-1}\setminus\{0\}$,
\begin{equation}\label{eq:walsh-layer}
 L_j\cap(L_j+d)=\varnothing.
\end{equation}
Indeed, if $(\xi,\eta)$ and $(\xi+d,\eta)$ belonged to $L_j$, then
$\eta$ would lie in both $T_{j,\sigma(\xi)}$ and
$T_{j,\sigma(\xi+d)}$. The labels are distinct and these supports
are disjoint.

Fix $d\ne0$ and choose $J$ with $d\in\Lambda_J$. Translation by
$d$ changes no block after $J$. Two frequencies differing by $d$
therefore cannot have different last nonzero blocks after $J$.
They cannot have the same such block either, by
\eqref{eq:walsh-layer}. Consequently
\begin{equation}\label{eq:walsh-final-overlap}
 E\cap(E+d)\subset E_J.
\end{equation}
This proves \eqref{eq:walsh-goal-thin}. Applying
Proposition~\ref{prop:walsh-criterion} gives strong small-2 for
\emph{all} measures in $M_E(W)$, while the nonzero singular $\mu$
shows that $E$ is not Riesz.

\subsection{Fourth powers of coefficients}
Equations~\eqref{eq:walsh-support} and \eqref{eq:walsh-new-coeff}
yield
\begin{align*}
 B_j-B_{j-1}
 &=a^4\sum_{\xi\in E_{j-1}}|c_\xi|^4
       \sum_\eta|\widehat h_{j,\sigma(\xi)}(\eta)|^4\\
 &\le a^4B_{j-1}v(r_j,q_j)\le\varepsilon_j.
\end{align*}
Since $B_0=1$ and old coefficients stabilize,
\[
 \sum_{\gamma\in\Lambda}|\widehat\mu(\gamma)|^4
 =\lim_jB_j\le1+\varepsilon.
\]
Fact~\ref{lem:fourier-l2} gives $\mu*\mu\in L^2(W)$.
Since $\widehat\mu(0)=1$, Parseval gives
\[
 \norm{\mu*\mu-m}_2^2
 =\sum_{\gamma\ne0}|\widehat\mu(\gamma)|^4\le\varepsilon.
\]
This completes the proof of Theorem~\ref{thm:walsh}.

\section{Algebraic reduction and the main theorems}\label{sec:algebra}
We now combine the model constructions. The reduction uses only
elementary properties of discrete abelian groups and the quotient
lifting of Lemma~\ref{lem:lift}.

\subsection{The asymmetric case}
\begin{lemma}\label{lem:algebra}
If $2\Gamma$ is infinite, at least one of the following holds:
\begin{enumerate}[label=\textup{(\roman*)},leftmargin=*]
\item $\Gamma$ contains an element of infinite order;
\item $\Gamma$ is torsion of unbounded exponent;
\item $\Gamma$ contains $\bigoplus_{\N}\Z/p\Z$ for an odd prime $p$;
\item $\Gamma$ contains $\bigoplus_{\N}\Z/4\Z$.
\end{enumerate}
\end{lemma}
\begin{proof}
After the first two cases, only torsion groups of bounded exponent
remain. Such a group is the direct sum of finitely many primary
components; we write $\Gamma_p$ for its $p$-primary component.
We first observe that a group $A$ annihilated by $p^r$
with finite $A[p]$ must be finite. This follows by induction on $r$:
$(pA)[p]\subset A[p]$, so $pA$ is finite by induction, and the kernel
of $p:A\to pA$ is the finite group $A[p]$.

If an odd primary component $\Gamma_p$ is infinite, its subgroup
$\Gamma_p[p]$ is therefore an infinite-dimensional vector space over
$\F_p$ and contains the subgroup in (iii). If every odd primary
component is finite, then $2\Gamma_2$ is infinite. Applying the
same observation to $2\Gamma_2$ shows that $(2\Gamma_2)[2]$ is an
infinite-dimensional vector space over $\F_2$.
Choose linearly independent $y_j$ in this space and $x_j\in\Gamma_2$
with $2x_j=y_j$. Each $x_j$ has order four. If
$\sum_jn_jx_j=0$, doubling first shows that all $n_j$ are even.
Writing $n_j=2m_j$, the independence of the $y_j$ then shows that
all $m_j$ are even. Thus every $n_j$ is divisible by four, and
$\langle x_j:j\ge1\rangle\cong\bigoplus_{\N}\Z/4\Z$.
\end{proof}

\begin{proof}[Proof of Theorem~\ref{thm:refined}\textup{(i)}]
In cases (i) and (ii) of Lemma~\ref{lem:algebra}, apply
Theorem~\ref{thm:RS}. In case (iii), use
Proposition~\ref{prop:odd-products}. The inclusion of
its dual subgroup in $\Gamma$ gives a quotient of $G$, and
Lemma~\ref{lem:lift} lifts the example to $G$. In case (iv), do
the same with the order-four construction in
Proposition~\ref{prop:carrier-products}.

Under each lift the full spectral space, singularity, mass and
$L^2$ square are preserved. If $E\subset\Delta\le\Gamma$, the
intersection $E\cap(\gamma-E)$ is empty for $\gamma\notin\Delta$,
and otherwise is the same intersection computed in $\Delta$.
Thus \eqref{eq:global-fi} is also preserved. All constructed sets
are countably infinite and asymmetric outside zero, hence proper.
\end{proof}

\begin{proposition}[The exact limit of the asymmetric criterion]
\label{prop:finite-double}
If $2\Gamma$ is finite and $E\subseteq\Gamma$ satisfies
$\#(E\cap(\gamma-E))<\infty$ for every $\gamma\in\Gamma$, then $E$
is finite. Consequently $2\Gamma$ is infinite if and only if
there exists a non-Riesz set satisfying this finite-intersection
condition.
\end{proposition}
\begin{proof}
For each $e\in E$, one has $e\in E\cap(2e-E)$. Therefore
\[
 E\subseteq\bigcup_{\delta\in2\Gamma}\bigl(E\cap(\delta-E)\bigr).
\]
When $2\Gamma$ is finite, the right-hand side is a finite union
of finite sets. Conversely, the preceding proof of
Theorem~\ref{thm:refined}(i) provides a non-Riesz example whenever
$2\Gamma$ is infinite.
\end{proof}

\subsection{The finite-doubling case and completion of the proof}
\begin{proof}[Proof of Theorem~\ref{thm:refined}\textup{(ii)}]
If $\Gamma$ is infinite and $2\Gamma$ is finite, then the kernel
$\Gamma[2]$ of the doubling map is infinite. As an $\F_2$-vector
space, it contains a subgroup
\[
 \Delta\cong\bigoplus_{\N}\F_2.
\]
The dual inclusion gives a quotient $q:G\to\widehat\Delta\cong W$.
Take $E\subset\Delta$ and $\mu$ from Theorem~\ref{thm:walsh} with
the prescribed $\varepsilon$, and lift them by Lemma~\ref{lem:lift}.
The lift is real, has the same mass and variation norm, and is
singular. By \eqref{eq:lift-fourier}, it identifies the full spectral
spaces, so strong small-2 is preserved. Further,
\[
 \norm{J\mu*J\mu-m_G}_2
 =\norm{J(\mu*\mu-m_W)}_2
 =\norm{\mu*\mu-m_W}_2.
\]
For $\gamma\notin\Delta$, $E\cap(E+\gamma)$ is empty; for
$\gamma\in\Delta\setminus\{0\}$, its finiteness is precisely
\eqref{eq:walsh-goal-thin}. The set remains proper and countably
infinite. This proves all assertions.
\end{proof}

\begin{proof}[Proof of Theorem~\ref{thm:main}]
The dual of an infinite compact abelian group is infinite.
The two alternatives in Theorem~\ref{thm:refined}, proved above,
therefore cover every infinite $G$ and give \eqref{eq:main}.
\end{proof}

\section*{Use of artificial intelligence}
This work was supported by OpenAI through the ChatGPT for Academic Researchers program.
AI-assisted tools were used during the preparation of this manuscript for mathematical verification and editorial support. The author has independently verified the content and takes full responsibility for its accuracy and conclusions.

\end{document}